\documentclass[11pt]{article}
\usepackage{amsmath, amssymb, amsthm, amsfonts, subfig}
\usepackage{graphicx}
\usepackage{hyperref}
\usepackage{verbatim}
\usepackage{float}
\usepackage{color,colortbl}

\usepackage[font=small,labelfont=bf]{caption}
\usepackage{titlesec}
\expandafter\def\expandafter\normalsize\expandafter{%
    \normalsize%
    \setlength\abovedisplayskip{5pt}%
    \setlength\belowdisplayskip{5pt}%
    \setlength\abovedisplayshortskip{-3pt}%
    \setlength\belowdisplayshortskip{3pt}%
}
\usepackage[scr=boondox,  % heavily sloped
            cal=esstix]   % slightly sloped
           {mathalpha}
\newtheorem{theorem}{Theorem}[section] % reset theorem counter every section

\newtheorem{corollary}{Corollary}[theorem] % reset corollary counter every theorem
\newtheorem{lemma}[theorem]{Lemma} % lemmas, propositions, claims, and conjectures all share numbering with theorems
\newtheorem{proposition}[theorem]{Proposition}
\newtheorem{conjecture}[theorem]{Conjecture}

\theoremstyle{definition}
\newtheorem{definition}{Definition}% definitions, examples, etc. have their own numbers
\newtheorem{example}{Example}

\theoremstyle{remark}
\newtheorem*{remark}{Remark} % remarks have no counter. 
\DeclareMathOperator{\outdeg}{outdegree}
\DeclareMathOperator{\indeg}{indegree}
\title{Chip Firing on Directed $k$-ary Trees}
\author{Ryota Inagaki \and Tanya Khovanova \and Austin Luo}
\date{}

\begin{document}

\maketitle

\begin{abstract}
Chip-firing is a combinatorial game played on a graph in which we place and disperse chips on vertices until a stable state is reached. We study a chip-firing variant played on an infinite rooted directed $k$-ary tree, where we place $k^\ell$ chips on the root for some positive integer $\ell$, and we say a vertex $v$ can fire if it has at least $k$ chips. A vertex fires by dispersing one chip to each out-neighbor. Once every vertex has less than $k$ chips, we reach a stable configuration since no vertex can fire. We determine the exact number and properties of the possible stable configurations of chips in the setting where chips are distinguishable.
\end{abstract}

\section{Introduction}

The game of chip-firing depicts a dynamical system and is an important part in the field of structural combinatorics. Chip-firing originates from problems such as the abelian sandpile \cite{dhar1999abelian}, which states that when a stack of sand grains exceeds a certain height, the stack will distribute grains evenly to its neighbors. Eventually, the sandpile may achieve a stable configuration, which is when every stack of sand cannot reach the threshold to disperse. This idea of self-organizing criticality combines a multitude of complex processes into a simpler process. Chip-firing as a combinatorial game on graphs began from the works such as those of Spencer \cite{MR856644}, Anderson, Lov\'asz, Shor, Spencer, Tardos, and Winograd \cite{zbMATH04135751} and Bj\"orner, Lov\'asz, and Shor \cite{MR1120415}. Many variants of the chip-firing game (see, for instance, \cite{MR3311336, MR3504984, MR4486679}) allow for the discovery of unique properties. For instance, in \cite{MR3311336, MR3504984}, certain classes of stable configurations can be described as a critical group. When the chips are distinguishable, numerous properties of chip firing with indistinguishable chips fail, prompting a new area of study.

\subsection{Unlabeled Chip-Firing on Directed Graphs}
\label{sec:unlabeledchipfiring}

 Unlabeled chip-firing occurs when indistinguishable chips are placed on vertices in a directed graph $G= (V, E)$. If a vertex has enough chips to transfer one chip to each out-neighbor, then that vertex can fire. In other words, if there are at least $\outdeg(v)$ chips on a vertex $v$, it can fire. When a vertex fires, it sends one chip to each neighbor and thus loses $\outdeg(v)$ chips. Once all vertices can no longer fire, we reach a \textit{stable configuration} (see Section~\ref{sec:prelim} for the full definition).

\begin{example}
    Figure~\ref{fig:exampleunlabel} shows the unlabeled chip-firing process when we start with $4$ chips at the root of an infinite binary tree.
\end{example}

\begin{figure}[H]
\centering
    \subfloat[\centering Initial configuration with $4$ chips]{{\includegraphics[width=0.35\linewidth]{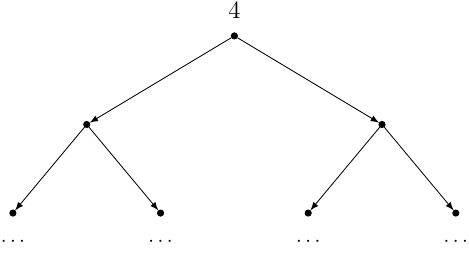} }}%
    \qquad
    \subfloat[\centering Configuration after firing once]{{\includegraphics[width=0.35\linewidth]{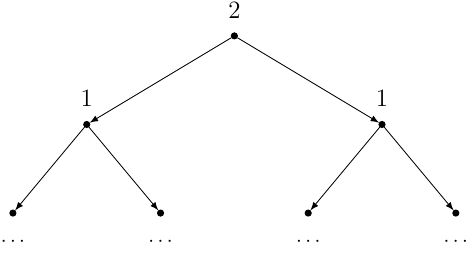} }}%
    \qquad
    \subfloat[\centering Configuration after firing twice]{{\includegraphics[width=0.35\linewidth]{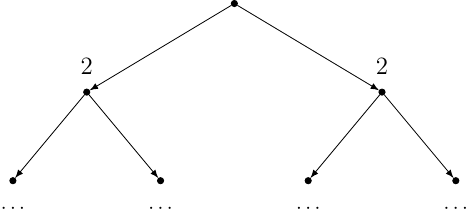} }}%
    \qquad
    \subfloat[\centering Stable configuration]{{\includegraphics[width=0.35\linewidth]{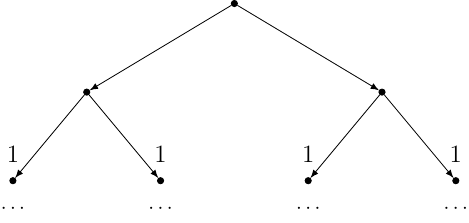} }}%
    \caption{Example of unlabeled chip-firing on an infinite directed, rooted binary tree}%
    \label{fig:exampleunlabel}
\end{figure}

Let us define a \textit{configuration} $c$ as a distribution of chips over the vertices of a graph, which is represented as a vector $\vec{c}$ in $\mathbb{N}^{|V|}$ (the set of infinite sequences indexed by the nonnegative integers whose entries are nonnegative integers), where the $k$th entry in $\vec{c}$ is the number of chips on vertex $v_k$ of the graph. One important property of directed graph chip-firing with unlabeled chips is the following property, which is analogous to the ``global confluence" property for chip-firing on undirected graphs (c.f. Theorem 2.2.2 of \cite{klivans2018mathematics}):

\begin{theorem}[Theorem 1.1 \cite{MR1203679}]\label{thm:GlobalConfluence}
    For a directed graph $G$ and initial configuration $c$ of chips on the graph, the unlabeled chip-firing game will either run forever or end after the same number of moves and at the same stable configuration. Furthermore, the number of times each vertex fires is the same regardless of the sequence of firings taken in the game.
\end{theorem}

\subsection{Labeled Chip-Firing on Directed Graphs}
\label{labeledchipfirng}

Labeled chip-firing is a variant of chip-firing where the chips are distinguishable. We denote this by assigning each chip a number from the set of $\{1,2,\dots, N\}$ where there are $N$ chips in total. A vertex $v$ can fire if it has at least $\outdeg(v)$ chips. When a vertex fires, we choose any $\outdeg(v)$ labeled chips and disperse them, one chip for each neighbor. The chip each neighbor receives may depend on the label of the chip. Labeled chip-firing was originally studied in the context of one-dimensional lattices \cite{MR3691530}. Labeled chip-firing has been studied on infinite binary trees when starting with $2^{\ell} - 1$ chips at the root for some $\ell \in \mathbb{N}$ (where $0 \in \mathbb{N}$) by Musiker and Nguyen \cite{musiker2023labeledchipfiringbinarytrees} and by the authors of this paper in \cite{inagaki2024chipfiringundirectedbinarytrees}.

In this paper, we study labeled chip-firing in the context of infinite directed $k$-ary trees for $k \geq 2$. Let us consider an infinite directed $2$-ary tree, or in other words, an infinite directed binary tree. Since each vertex $v$ has $\outdeg(v) = 2$, a vertex can fire if it has two chips. When a vertex fires, we arbitrarily select two chips and send the smaller chip to the left child and the larger one to the right. Note that when we say a chip is smaller or larger than another chip, we refer to the numerical values of the labels assigned to them. The mechanics of chip firing on $k$-ary trees is a straightforward generalization of the above.

In labeled chip-firing, Theorem~\ref{thm:GlobalConfluence} does not hold. This means that we can achieve different stable configurations depending on the sets of chips we arbitrarily select to fire. More precisely, two stable configurations would always have the same number of chips at each vertex, but the labels might differ.

% \AL{I feel like we should move 2.2 to the introduction}

\begin{example}
    Consider a directed binary tree with $4$ labeled chips: $(1,2,3,4)$ at the root. Notice that since chips are only sent along directed edges, once a chip is sent to the left or right, it cannot go back. Therefore, if we fire the pair of chips $(1,2)$ first, we end up with a different stable configuration than if we fire the pair $(2,3)$ first. Figure~\ref{fig:confluencebreak} illustrates this initial firing. 
\end{example}

\begin{figure}[H]
\centering
    \subfloat[\centering Configuration after firing $(1,2)$]{{\includegraphics[width=0.35\linewidth]{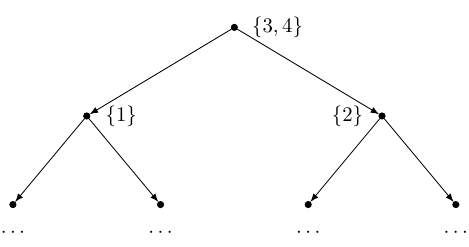} }}%
    \qquad
    \subfloat[\centering Configuration after firing $(2,3)$]{{\includegraphics[width=0.35\linewidth]{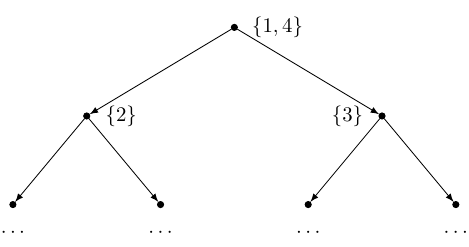} }}%
    \caption{Example of confluence breaking}%
    \label{fig:confluencebreak}
\end{figure}
   
Therefore, to obtain certain stable configurations, we pick certain sets of chips to fire, which is seen in Section~\ref{sec:PossibleSequences}. Thus, we are motivated to study the properties of labeled chip-firing, such as the number of stable configurations and where certain orders of labeled chips can appear in the stable configurations.

\subsection{Objectives and Roadmap}
Our problems are similar to those studied by Musiker and Nguyen in \cite{musiker2023labeledchipfiringbinarytrees}, but in the context of labeled-chip firing on directed $k$-ary trees, we can ask: 
\begin{itemize}
    \item How many different stable configurations are there? 
    \item What does the stable configuration look like? 
\end{itemize}

In Section~\ref{sec:prelim}, we more precisely introduce labeled chip-firing on directed $k$-ary trees and provide important definitions describing our setup. In Section~\ref{countingfinal}, we find the number of possible final stable configurations for a directed $k$-ary tree starting with $k^{\ell}$ labeled chips at the root in terms of $C_{k, m}$, the $m$th $k-$dimensional Catalan number. In Section~\ref{sec:PossibleSequences}, we prove general results on what configurations are possible. In particular, in a stable configuration, only vertices on layer $\ell+1$ have chips. Moreover, each vertex has exactly one chip. Thus, each stable configuration corresponds to a permutation. 

In Section~\ref{sec:DigitReversalPermutation}, we introduce the digit-reversal permutation $R_k(\ell)$, and we prove that it describes an attainable stable configuration $Z_k(\ell)$. In Section~\ref{sec:NumberOfInversions}, we find that the permutation corresponding to $Z_k(\ell)$ has the largest possible number of inversions, which is \[\frac{k^{2\ell}-\ell k ^{\ell+1}+(\ell-1) k^{\ell}}{4},\] among all permutations corresponding to stable configurations. In Section~\ref{sec:LongestDecreasingSubseq}, we find that the longest decreasing subsequence in $Z_k(\ell)$ has length $(k+1)k^{\ell/2-1}-1$ if $\ell$ is even and $2k^{(\ell-1)/2}-1$ if $\ell$ is odd. Finally, in Section~\ref{sec:LongestDecreasingSequencesGeneral}, we use our results from Section~\ref{sec:LongestDecreasingSubseq} to prove a lower and an upper bound for the longest possible decreasing subsequence that can appear in a stable configuration.

\section{Definitions and Basic Results}
\label{sec:prelim}

\subsection{Definitions}
\label{sec:definitions}

In this paper, we consider infinite rooted directed $k$-ary trees as our underlying graphs. 

In a \textit{rooted tree}, we denote one distinguished vertex as the \textit{root} vertex $r$. Every vertex in the tree, excluding the root, has exactly one parent vertex. A vertex $v$ has \textit{parent} $v_p$ if there is a directed edge $v_p \to v$. If a vertex $v$ has parent $v_p$, then vertex $v$ is a \textit{child} of $v_p$.

An \textit{infinite directed $k$-ary} tree is defined as an infinite directed rooted tree where each vertex has $\outdeg$ $k$ and $\indeg$ $1$ (except the root, which has $k$ children but zero parents). The edges are directed from a parent to children.

We define the \textit{initial state} of chip-firing as placing $N$ chips on the root where, in the case of labeled chip-firing, they are labeled $1,2, \dots, N$. A vertex $v$ can \textit{fire} if it has at least $\outdeg(v)=k$ chips. When vertex $v$ \textit{fires}, it transfers a chip from itself to each of its $k$ neighbors. In the setting of labeled chip firing, when a vertex fires, it chooses and fires $k$ of its chips so that among those $k$ chips, the one with the $ith$ smallest label gets sent to the $i$th leftmost child from the left. A \textit{strategy} is a procedure dictating an order in which $k$-tuples of chips on a vertex get fired from which vertex. In this paper, we assume $k \geq 2$ since if $k = 1$ and the tree has any positive number of chips, then the chip-firing process can continue indefinitely.

We define a vertex $v_j$ to be on \textit{layer} $i+1$ if the path of vertices traveled from the root to $v_j$ traverses $i$ vertices. Thus, the root $r$ is on layer $1$.

The following is the labeling procedure for vertices. Let us take the set of vertices on a layer $\ell$. We label the vertices $v_i, v_{i+1}, \dots, v_{i+j}$ on layer $\ell$ where $i = \frac{k^{\ell-1}-1}{k-1} + 1$ and $j =k^{\ell-1}-1$.  For vertex $v_i$, the $j$th leftmost child has label $v_{k(i-1)+j+1}$. Figure~\ref{figures/directedlabeleingex.pdf} represents the labeling for the first $2$ layers in the directed $5$-ary tree. 

\begin{figure}[H]
    \centering
    \includegraphics[width=0.4\linewidth]{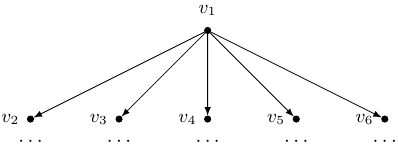}
    \caption{Labeling for $2$ layers in directed $5$-ary tree}
    \label{figures/directedlabeleingex.pdf}
\end{figure}

We denote the \textit{straight left descendant} of a vertex $v_n$ as any vertex $v_j$ where $j > n$ such that if we take the path of vertices from $v_j$ to $v_n$, each vertex on the path traversed is the left-most child of their parent and the \textit{straight right descendant} is defined similarly. If the straight left descendant of a vertex $v_n$ is on the last layer with chips in the stable configuration, it is called the \textit{bottom straight left descendant} of $v_n$, and the \textit{bottom straight right descendant} is defined similarly. A vertex $v_n$ is a \textit{top straight ancestor} of vertex $v_j$ if vertex $v_n$ is the left-most child of its parent and vertex $v_j$ is a straight right descendent of vertex $v_n$ or vice versa. In the case of the root, it is considered the top straight ancestor of the left and right descendants. 

The \textit{stable configuration} is a distribution/placement of chips over the vertices of a graph such that no vertex is able to fire. In this paper, we write each stable configuration as a permutation of $1, 2, \dots, k^{\ell}$, which is the sequence of chips in the $(\ell+1)$st layer of the tree in the stable configuration read from left to right. This is because, as we will see in the next subsection, the stable configuration will have one chip at each vertex in layer $\ell+1$. This is our convention for the rest of this paper (for instance, the stable configuration in Figure~\ref{fig:directedex} would be denoted by permutation/sequence $1, 2, 3, 4$).

\subsection{Unlabeled Chip-Firing on Directed \texorpdfstring{$k$}{k}-ary Trees}
\label{sec:unlabeled}

We first examine properties of unlabeled chip-firing on infinite directed $k$-ary trees (i.e., ignoring labels) when starting with $k^\ell$ chips at the root where $\ell \in \mathbb{N}^+$. As the stable configuration and the number of firings do not depend on the order of firings, we can assume that we start from layer 1 and proceed by firing all the chips on the given layer before going to the next layer. Thus, for each $t \in \{1,2, \dots, \ell \}$, each vertex on layer $t$ fires $k^{\ell-t}$ times and sends $k^{\ell-t}$ chips to each of its children. In the stable configuration, each vertex on layer $\ell + 1$ has exactly $1$ chip, and for all $i \neq \ell + 1$, the vertices on layer $i$ have $0$ chips.

\subsection{Labeled Chip-Firing on Directed \texorpdfstring{$k$}{k}-ary Trees}
\label{sec:labeled}

 We now give an example of a labeled chip-firing game on the directed $k$-ary tree for $k=2$ and establish a useful lemma on the positions of the smallest and largest chip in the stable configuration of labeled chips.
\begin{example}
Consider again a directed binary tree with $4$ labeled chips: $(1,2,3,4)$ at the root. Figure~\ref{fig:directedex} shows a possible sequence of firings that stabilizes a binary tree starting with $4$ labeled chips at the root. 
\end{example}
\begin{figure}[H]
    \centering
    \subfloat[\centering Initial configuration with $4$ chips]{{\includegraphics[width=0.4\linewidth]{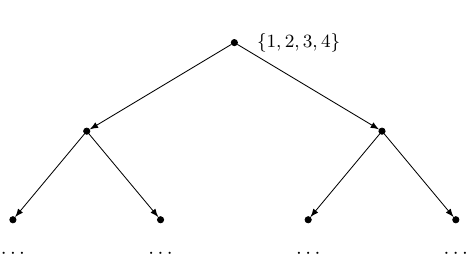} }}
    \qquad%
    \subfloat[\centering State after firing root once]{{\includegraphics[width=0.4\linewidth]{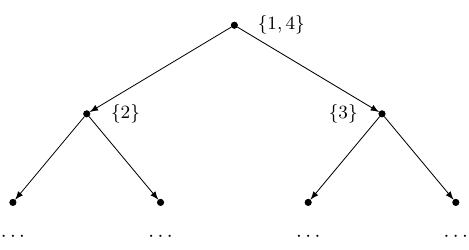} }}%
    \qquad
     \subfloat[\centering State after firing root a second time]{{\includegraphics[width=0.4\linewidth]{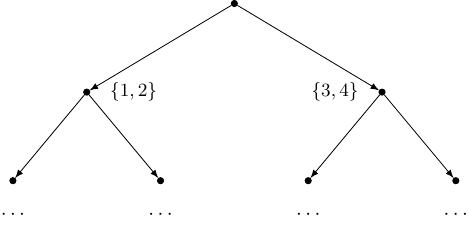} }}%
    \qquad
    \subfloat[\centering Stable configuration]
    {{\includegraphics[width=0.4\linewidth]{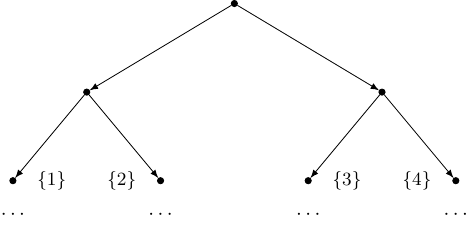} }}%
    \caption{Example of labeled chip-firing in a directed binary tree with $4$ chips}%%
    \label{fig:directedex}
\end{figure}

In the previous example, observe that given that a vertex fires a set of ordered pairs of labeled chips, any order in which those pairs of chips are fired yields the same distribution of chips to the children. This is a fact that holds in general: in chip-firing on directed $k$-ary trees, given that a vertex fires a set of $k$-element tuples of labeled chips, any order in which those $k$-tuples of chips are fired yields the same distribution of chips to the children.

We conclude the section with the final positions of the chips with the smallest and largest labels.

\begin{lemma}\label{lem:BottomStraightLeftStraightRight}
    If we start with $k^{\ell}$ labeled chips at a vertex $v$, then, in the stable configuration, the bottom straight left descendant and bottom straight right descendant of any vertex $v$ contain the smallest and largest chips, respectively, in the subtree with root $v$.
\end{lemma}

\begin{proof}
    Let $S$ denote the set of chips on the root $v$ of a subtree before this vertex starts firing. Let chip $c_1$ be the smallest labeled chip and chip $c_2$ be the largest labeled chip in $S$. No matter which $k$-tuple of chips we select to fire, chip $c_1$ will always be sent to the left since it is the smallest labeled chip, and chip $c_2$ will always be sent to the right since it is the largest labeled chip. Therefore, for some vertex $v$, in the subtree with the root at $v$, the bottom straight left descendant and bottom straight right descendant of vertex $v$ will contain the smallest and largest chips in that subtree, respectively.
\end{proof}

\section{Counting the Number of Stable Configurations}
\label{countingfinal}
In \cite{musiker2023labeledchipfiringbinarytrees}, one unanswered question is the number of possible stable configurations when starting with $2^\ell - 1$ labeled chips in an undirected binary tree. We answer this question in the directed $k$-ary tree setting. We find a bijection between the number of ways to sort the labeled chips to the $k$ children and the collection of certain lattice walks.

Let us have a vector space $\mathbb{R}^k$, where $\vec{e}_i$ is an $i$th elementary basis vector. We can consider $\vec{e}_i$ as one possible step on a \textit{walk}. Also, let us denote by $\vec{1}$ the vector in $\mathbb{R}^{k}$ with all entries being $1$.
\begin{definition}
      Define $A_{k,m}$ to be the collection of all walks in $\mathbb{R}^k$ of length $km$ starting at the origin, where $\vec{a}_i$ is the $i$-th step. All walks end at point $(m,m,\ldots,m)$, or equivalently $\sum_{j=1}^{km} \vec{a}_j = m \vec{1}$. In addition, the walks have the ballot property, where for all $1 \leq i \leq km$, the intermediate point on the walk $(x_1, x_2, \dots, x_{k}) = \sum_{j=1}^i \vec{a}_j$ is such that $x_1 \geq x_2 \geq ... \geq x_k$.
\end{definition}

\begin{example}
    A walk in $A_{2, 4}$. Figure~\ref{fig:2arywalk} illustrates a walk of length $8$ in $\mathbb{R}^2$. Each horizontal step is $\vec{e}_1$, and each vertical step is $\vec{e}_2$.
\end{example}
\begin{figure}[H]
    \centering
    \includegraphics[width=0.3\linewidth]{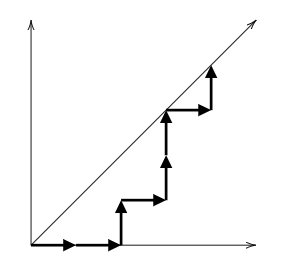}
    \caption{Example of a walk of length $8$ in $\mathbb{R}^2$ with the diagonal $y =x$}
    \label{fig:2arywalk}
\end{figure}
  
We show that there is a bijective mapping between $A_{k,k^{\ell-1}}$ and the ways of dispersing $k^\ell$ chips.  The cardinality of $A_{k,k^{\ell-1}}$ is the $k^{\ell-1}$th $k$-dimensional Catalan number \cite{MR1069169}. We denote such a number as $C_{k,m}$ where $k$ is the dimension and $m$ is the index. Thus, the cardinality of $A_{k,k^{\ell-1}}$ is $C_{k,k^{\ell-1}}$.
  
\begin{lemma}
\label{walktocat}
    If we start with $k^\ell$ labeled chips at the root, then the number of ways to disperse $k^\ell$ chips to $k$ root's children is exactly $C_{k,k^{\ell-1}}$.
\end{lemma}

\begin{proof}
    We outline a procedure for constructing a walk given a dispersion of chips. If the chip labeled $i$ ends in the $j$-th leftmost child, then on the $i$-th step $\vec{a}_i = \vec{e}_j$. 

    We first observe that, indeed, this procedure maps from the possible ways of dispersing $k^{\ell}$ chips to $k$ children to walks in $A_{k, k^{\ell - 1}}$. Suppose for the sake of contradiction that for some dispersion of chips, the procedure outputs a walk $(\vec{a}_1, \vec{a}_2, \dots, \vec{a}_{k^{\ell}})$ not in $A_{k, k^{\ell - 1}}$. This means that either $\sum_{j=1}^{k^{\ell}} \vec{a}_j \neq k^{\ell-1} \vec{1}$ or for some $i \in [k^{\ell}]$  there is some $s \in [k-1]$ such that for $(x_1,  x_2, \dots, x_{k}) = \sum_{m=1}^{k^{\ell}} \vec{a}_m$, for $x_s < x_{s+1}$. If the former holds, this implies that in the dispersion, one child got more chips than the others, which cannot happen as each firing disperses one chip to each of the $k$ children. If the latter holds, this means that among the $i$th smallest chips, more of them got sent to the $(s+1)$st child from the left than to the $s$th child from the left, contradicting our firing rules.
    
    We show how to reconstruct the dispersion given a walk in $A_{k, k^{\ell-1}}$. Suppose the first appearance of $\vec{e}_j$ direction happens at $y_j$. Then, during the first firing, we send chip $y_j$ to $j$th vertex from the left. As at each moment the coordinates on the path decrease, we know that if $j_1 < j_2$, then the first step $\vec{e}_{j_1}$ happened before the first step in direction $\vec{e}_{j_2}$. Thus, $y_{j_1} < y_{j_2}$, and our firing assignment is legitimate. We similarly assign consecutive firings.
    
  Therefore, we have constructed a bijection between the dispersion of chips and $A_{k, k^{\ell-1}}$. Thus, the number of ways of dispersing chips is $C_{k,k^{\ell-1}}$.
\end{proof}

The closed form of $C_{k,m}$ was derived in \cite{MR1069169} (see also entry A060854 in \cite{oeis}) and is represented as:
\[\frac{\binom{km}{m,m,\dots,m}}{\binom{m + 1 }{ m} \binom{m+2 }{ m} \cdots \binom{m + k-1 }{ m}} .\] Notice that $C_{2, m}$ is equal to the $m$th Catalan number $C_m = \frac{1}{m+1}\binom{2m}{m}.$

Recall from Section~\ref{sec:unlabeledchipfiring} that, since the root fires until it no longer can, we have $k$ subtrees with $k^{\ell-1}$ chips at their roots $v_2, v_3, \dots v_{k+1}$. Thus, we can find the number of stable configurations recursively. We do this in Theorem~\ref{kdimensionalcatalan}. Let $\kappa(\ell)_k$ denote the number of stable configurations when starting with $k^\ell$ labeled chips at the root of the directed $k$-ary tree.
     
\begin{theorem}
\label{kdimensionalcatalan}
    The number of stable configurations when starting with $k^\ell$ labeled chips at the root, $\kappa(\ell)_k$, can be calculated recursively as 
    $$\kappa(\ell)_k = C_{k,k^{\ell-1}} C_{k,k^{\ell-2}}^{k^1}C_{k,k^{\ell-3}}^{k^2}C_{k,k^{\ell-4}}^{k^3}\cdots C_{k,k^1}^{k^{\ell-2}}C_{k,k^0}^{k^{\ell-1}}.$$
\end{theorem}
    
\begin{proof}
Once the root fires until it no longer can, we will have $k$ subtrees with $k^{\ell-1}$ chips at their roots $v_2, v_3, \dots, v_{k+1}$. Each subtree has $\kappa(\ell - 1)_k$ possible stable configurations. Thus, we have the recursive relation of $\kappa(\ell)_k = C_{k,k^{\ell-1}}\kappa(\ell - 1)_k^k$.  Let us prove the theorem statement by induction. When $\ell =1$, there is only $1$ configuration which is equal to $C_{k,1} = 1$. Let us assume that: 
$$\kappa(\ell)_k = C_{k,k^{\ell-1}} C_{k,k^{\ell-2}}^{k^1}C_{k,k^{\ell-3}}^{k^2}C_{k,k^{\ell-4}}^{k^3}\cdots C_{k,k^1}^{k^{\ell-2}}C_{k,k^0}^{k^{\ell-1}},$$
as the induction hypothesis. 

By the recursive relationship and the induction hypothesis, we have
            \begin{align*} 
            \kappa(\ell+1)_k &= C_{k,k^\ell}\left(C_{k,k^{\ell-1}} C_{k,k^{\ell-2}}^{k^1}C_{k,k^{\ell-3}}^{k^2}C_{k,k^{\ell-4}}^{k^3} \cdots C_{k,k^1}^{k^{\ell-2}}C_{k,k^0}^{k^{\ell-1}}\right)^k \\
             &= C_{k,k^{\ell}} C_{k,k^{\ell-1}}^{k^1}C_{k,k^{\ell-2}}^{k^2}C_{k,k^{\ell-3}}^{k^3} \cdots C_{k,k^1}^{k^{\ell-1}}C_{k,k^0}^{k^{\ell}},
            \end{align*}
            which concludes the proof. 
\end{proof}

\section{Sequences that Can Appear in the Stable Configuration}\label{sec:PossibleSequences}
In this section, we discuss possible stable configurations resulting from labeled chip-firing on a $k$-ary tree starting with $k^{\ell}$ labeled chips and introduce the stable configuration $Z_k(\ell)$. We denote the corresponding permutation as $Z_k(\ell)$ too. We will show later this permutation the maximal number of inversions when viewed as a sequence of chips in the $(\ell+1)$st row.

In addition to determining the positions where chips can end up, one can describe the possible stable configurations by finding that certain permutation patterns can or cannot appear.          
\begin{definition}
        Given a permutation $\sigma = \sigma_1, \sigma_2, \sigma_3, \dots, \sigma_n \in S_n$, we say that a subsequence $w=\{w_1,w_2,w_3,\dots,w_n\}$ has \textit{a permutation pattern} $\sigma$ if there are indices $i_1 < i_2< \dots < i_n$ such that for all $j \in \{1, 2, \dots, n\}$ of $w$, the term $w_{i_j}$ is the $\sigma_j$th smallest term in the subsequence $w_{i_{1}},w_{i_{2}},\dots,w_{i_{n-1}},w_{i_{n}}$. 
\end{definition}

\begin{example}
A permutation pattern of subsequence $3,5,6,9$ is $1,3,2,4$.
\end{example}
The study of permutation patterns is a growing area of interest in enumerative combinatorics, as seen in \cite{MR4596199}.

Consider a chip-firing strategy $F$ on chips 1 through $k^\ell$ at the root in a directed $k$-ary tree. Suppose we have a different situation with $k^\ell$ chips, where the labels are distinct but might not start with one and might have gaps. As we only care about the order of the labels, we can apply the same strategy $F$ to the new situation. If on chips $1$ through $k^\ell$ the stable configuration corresponds to permutation $\sigma$, then in the latter case, the stable configuration corresponds to a permutation with pattern $\sigma$. 

\begin{example}
Consider a directed binary tree. If strategy $F$ leads to the stable configuration $1,3,2,4$ on the standard set of chips, then on chips $3,5,6,9$, the same strategy $F$ leads to the stable configuration $3,6,5,9$.
\end{example}

Given a firing strategy $F$ on $k^n$ labeled chips, we can define new strategies on $mk^{n}$ labeled chips, where we divide the chips into $m$ groups of $k^n$ chips each and apply $F$ to each group independently. Here, we describe one such special strategy, which we call a $F$-\textit{bundle}. We divide the chips at the root into $m$ groups of chips, such that the $j$th group contains all chips of the form $(im + j)$, where $0 \leq i < k^n$. After applying the strategy $F$ to the elements in each group, a vertex on layer $n+1$ will get $m$ consecutive chips. The vertex that would have received $i$ with strategy $F$ in the chip-firing system starting with $k^n$ chips receives the chips $(i-1)m + 1$ through $im$.

\begin{example}
\label{ex:identity}
    Suppose $F_{id}$ is a firing strategy on one vertex of a $k$-ary tree with $k$ chips. Consider a $F_{id}$-bundling strategy at the root and at every other layer except the last. The stable configuration of such firing is the identity permutation.
\end{example}

The bundling strategy corresponds to a notion of inflated permutation. These definitions were first introduced in \cite{MR2170110} and used in \cite{MR4245257}.

\begin{definition}[Inflation]
Given a permutation $\tau$ of length $n$, the \textit{inflation} of $\tau$ with a sequence of $n$ permutations $\gamma_1, \ldots, \gamma_n$ is a permutation $\tau'$ of length $|\gamma_1| + \cdots + |\gamma_n|$ that consists of $n$ blocks, such that the $i$-th block is order-isomorphic to $\gamma_i$, and any restriction of $\tau'$ to one element in each block is order-isomorphic to $\tau$. This inflation is denoted as $\tau[\gamma_1, \ldots, \gamma_n]$.
\end{definition}

We are interested in the case when all permutations $\gamma_i$ are of the same length. In this case, the length of the inflation is the length of $\tau$ times the length of each block $\gamma$.

\begin{example}
If $\tau = 231$, $\gamma_1 = 21$, $\gamma_2 = 12$, and $\gamma_3 = 21$, we get the inflation $231[21, 12, 21] = 43\ 56\ 21$, where we add spaces for clarity. Note how each element in the original permutation $\tau$ corresponds to a block of elements in the inflation.
\end{example}

\begin{proposition}\label{prop:BundleBlockPerms}
Suppose we start with $k^{n+m}$ labeled chips and fire the $F$-bundle from the root, where $F$ is a strategy of firing $k^n$ chips that leads to the stable configuration with permutation $\tau$. Then we fire strategy $F_i$ from $i$th leftmost vertex in the $(n+1)$st layer, where applying strategy $F_i$ leads to the stable configuration with permutation $\gamma_i$. Then, our overall strategy leads to the stable configuration with permutation, which is inflation $\tau[\gamma_1, \ldots, \gamma_{k^n}]$.
\end{proposition}

\begin{proof}
Consider collections of chips $S_{\ell} = \{\ell k^m+1, \ell k^m+2, ....., (\ell+1)k^m\}$ defined for $\ell \in \{0, 1, ..., k^{n}-1\}.$ In performing the $F$-bundling, we treat each $S_{\ell}$ as a single chip $\ell$ and then apply strategy $F$ to obtain state $\tau.$ Now replace each $\ell$ with $\{\ell k^m+1, \ell k^m+2, ....., (\ell+1)k^m\}.$ Observe that if we restrict the state of the chips on the tree to one chip per vertex, we obtain that the resulting sequence is order isomorphic to $\tau.$ 

Now, consider for each $i$ applying strategy $F_i$ to the subtree rooted at the vertex that is the $i$th from the left in layer $n+1$. This ensures that in the stable configuration of chips, the sequence of chips that are in the subtree rooted by the $i$th vertex from the left in the $(n+m+1)$st layer is order isomorphic to $\gamma_i.$ Combining with the previous paragraph, this implies that our stable configuration is inflation $\tau[\gamma_1, \gamma_2, ..., \gamma_{k^n}].$
\end{proof}

There is a particular case of inflation that is often used. Rather than specifying the inflation of a permutation $\tau$ with a sequence of $|\tau|$ different permutations $\gamma_1, \ldots, \gamma_n$, we have a special notation for the case when all of these permutations are the same, i.e., $\gamma_1 = \cdots = \gamma_n$.

\begin{definition}[Tensor product]
Given two permutations $\tau \in S_n$ and $\gamma \in S_m$, their \emph{tensor product} $\tau[\gamma]$ is a permutation of length $mn$ that consists of $n$ blocks of length $m$, where each block is order-isomorphic to $\gamma$, and the restriction of $\tau[\gamma]$ to one element in each block is order-isomorphic to $\tau$.
\end{definition}

\begin{example}\label{ex:987654321}
Consider $\tau = \gamma = 321$, then $\tau[\gamma] = 987\ 654\ 321$ is the decreasing permutation in $S_9$. If $k=2$, and we start with 8 chips on a directed binary tree, there exists a strategy $F$ that can get us a stable configuration with a permutation that has a subsequence order isomorphic to 321, which is shown in Example~\ref{ex:unbundling}. It follows that if we perform the $F$-bundle on 64 chips and then repeat the $F$ strategy on each subtree, we can get a permutation order isomorphic to 987654321.
\end{example}

Now, we define another strategy that is similar to bundling but is opposite in some sense. Given a firing strategy $F$ on $k^n$ labeled chips, we call our new strategy on $mk^{n}$ labeled chips an $F$-\textit{unbundle}. We divided the chips at the root into $m$ groups of chips, such that each group contains a set of chips from $ik^n+1$ to $(i+1)k^n$, for $0 \leq i < m$. After applying the strategy $F$ at the root to each group of chips, on layer $n$, we get $k^n$ vertices each with $m$ chips. The vertex that would have received $i$ with strategy $F$ receives the chips $i, i+k^n$, $i+2\cdot k^n$, $i+3\cdot k^n$, $\dots$, $i+(m-1)k^n$.

\begin{definition}\label{def:zStrategy}
Let $F_{id}$ be a firing strategy on one vertex with $k$ chips. Consider an $F_{id}$-unbundling strategy at the root and at every other vertex except in the last layer. As a result, our stable configuration is a special permutation, which is extreme in some senses. We denote this permutation as $Z_k(\ell)$ or $Z$ when $k$ and $\ell$ are clear. Also, we sometimes refer to the corresponding stable configuration as $Z_k(\ell)$.
\end{definition}

\begin{example}
\label{ex:unbundling}
        Suppose we have 8 chips at the root of a binary tree. We use the $F_{id}$-unbundling strategy at every node to get to the $Z_2(3)$ permutation. First, we fire pairs $(1, 2)$, $(3, 4)$, $(5, 6)$, and $(7, 8)$. Then, on the left child of the root, we fire $(1, 3)$ and $(5, 7)$, and on the right child of the root, we fire $(2, 4)$ and $(6, 8)$. We obtain in the end the permutation $Z_2(3)$: 1, 5, 3, 7, 2, 6, 4, 8. Figure~\ref{fig:z23ub} illustrates the complete firing process. 
\end{example}

\begin{figure}[H]
    \centering
    \subfloat[\centering Initial configuration with $8$ chips]{{\includegraphics[width=0.4\linewidth]{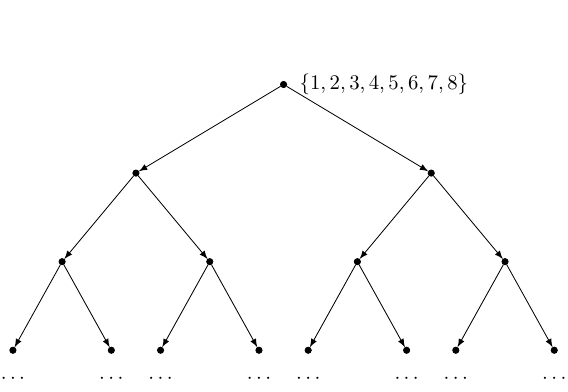} }}
    \qquad%
    \subfloat[\centering State after firing pairs: $(1,2), (3,4), (5,6), (7,8)$]{{\includegraphics[width=0.4\linewidth]{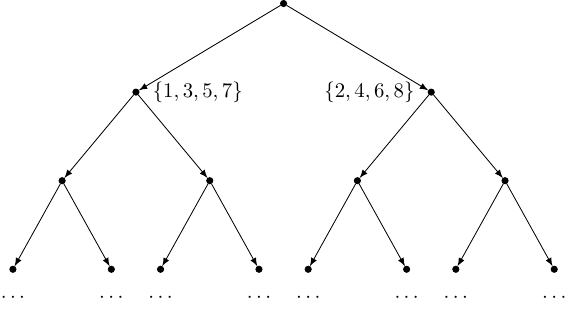} }}%
    \qquad
     \subfloat[\centering State after firing pairs: $(1,3), (5,7), (2,4), (6,8)$]{{\includegraphics[width=0.4\linewidth]{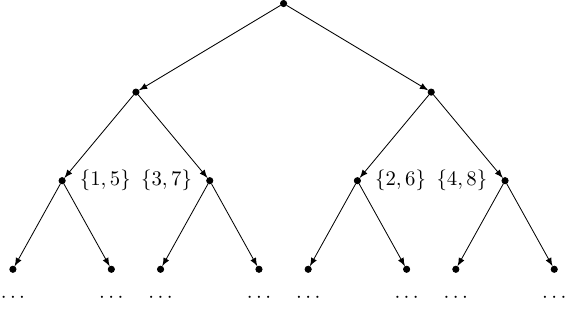} }}%
    \qquad
    \subfloat[\centering Stable configuration]
    {{\includegraphics[width=0.4\linewidth]{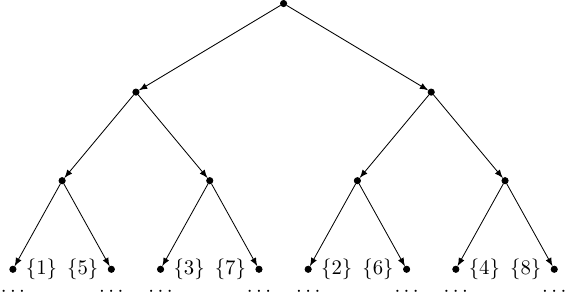} }}%
    \caption{Firing process to obtain the $Z_2(3)$ permutation}%%
    \label{fig:z23ub}
\end{figure}

\begin{theorem}
\label{thm:anyperm}
    For any permutation pattern $P$ of length $k^n$, one can construct $P$ as a permutation pattern of a subsequence in a stable configuration through the firing of $k^{2n}$ chips starting at the root. 
\end{theorem}

\begin{proof}
    Consider permutation $P$ and strategy $F$ for obtaining the identity permutation in the $(n+1)$st layer when starting with $k^n$ labeled chip. Now, when we have $k^{2n}$ chips, we apply $F$-unbundling at the root. Thus, each vertex on layer $n+1$ has chips $c, c+k^{n}, c+k^{n+1}, \dots , c+k^{2n}-k^{n}$, where $c$ is the chip that the vertex on layer $k^n$ would have received when applying strategy $F$. We now apply any strategy to finish the firings. To obtain permutation pattern $P = p_1p_2 \dots p_{k^n}$ from the resulting stable configuration, we take chips $i + (p_i - 1)k^n$ for all $i \in [1,k^n]$. Since $1 + (p_1 - 1)k^n, 2+ (p_2 - 1)k^n, \dots, k^n + (p_{k^n}-1)k^n$ is a subsequence of the stable configuration and since $1 + (p_1 - 1)k^n, 2+ (p_2 - 1)k^n, \dots, k^n + (p_{k^n}-1)k^n$ has the same relative order as $p_1,p_2,\dots, p_{k^n}$, we obtain that the permutation pattern $P = p_1p_2\dots p_{k^n}$ appears in the stable configuration.
\end{proof}

\section{The Digit-Reversal Permutation}\label{sec:DigitReversalPermutation}

 Since the order of which chips are fired at each vertex does not matter, we need a clever strategy to create a permutation with a lot of inversions at the stable configuration. As we know, we can get an identity permutation at the stable configuration by using $F_{id}$-bundling at each vertex. However, the stable configuration cannot have chips in decreasing order, as, for example, the first chip is always labeled one. In this section, we study the digit-reversal permutation and show that it describes the stable configuration with the largest possible number of inversions.

We define the digit-reversal permutation $R'_k(\ell):$
\begin{definition}
    A \textit{radix-$k$ digit-reversal permutation} $R_k'(\ell)$ is a permutation of $k^\ell$ numbers from 0 to $k^\ell-1$. We represent each integer from $0$ to $k^{\ell}-1$ in base $k$ and prepend it with zeros, so each number becomes a string of length $\ell$. After that, we map each number to the number whose representation has the same digits in the reversed order \cite{MR900058}. We define $R_k(\ell)$ to denote the permutation of $1, 2, \dots, k^{\ell}$ where we add $1$ to each term in $R_k'(\ell).$
\end{definition}
For $k=2$, the digit-reversal permutation is often called a \textit{bit-reversal permutation}. It is the same as the sequence consisting of the $2^{\ell}$th to $(2^{\ell+1}-1)$st element of sequence A030109 in the OEIS \cite{oeis}.

\begin{example}\label{ex:321}
    We compute the bit-reversal permutation $R_2'(3)$ of length $2^3$. First, we begin with $0, 1, 2, ..., 2^3-1$. We write the numbers in $000, 001, 010, 011, 100, 101, 110, 111.$ Then we reverse the bits to obtain $000, 100, 010, 110, 001, 101, 011, 111$, which in decimal are $0, 4, 2, 6, 1, 5, 3, 7.$ Therefore, we obtain $R_2'(3)$ to be $0, 4, 2, 6, 1, 5, 3, 7.$ By adding $1$ to each term in $R_2'(3)$, we obtain that permutation $R_2(3)$ is 1, 5, 3, 7, 2, 6, 4, 8. We observe that this is the same permutation as $Z_2(3)$. We show that this is not a coincidence in the next section.
\end{example}

We now prove that $R_k(\ell)$ is an attainable stable configuration of $k^{\ell}$ chips on a $k$-ary directed tree. Recall from Definition \ref{def:zStrategy} that $Z_k(\ell)$ is the permutation representing the stable configuration resulting from the $F_{id}$-unbundling strategy at the root and at every other vertex except in the last layer. We show that $Z_k(\ell)$ is $R_k(\ell)$.

\begin{proposition}\label{prop:k-itReversedPermutation} The permutation $Z_k(\ell)$ is $R_k(\ell)$, i.e., the permutation of $1, 2, \dots, k^{\ell}$ resulting from adding $1$ to the radix-$k$ digit-reversal permutation $R'_k(\ell)$ of $k^{\ell}$ elements $0, 1, 2, \dots, k^{\ell}-1$.
\end{proposition}

\begin{proof}
    Suppose we have chips labeled $0, 1, 2, \dots, k^{\ell}-1$ at the root. After performing $F_{id}$-unbundling at the root, the $i$th child will receive the chips with labels that end in $i-1$ in base-$k$. Similarly, the $F_{id}$-unbundling at the next layer will sort the chips by the second to last digit. In the end, the stable configuration will correspond to the radix-$k$ digit-reversal permutation. Increasing the chips by 1 concludes the proof.
\end{proof}

\begin{remark}
    The algorithm used to create stable configuration $Z_k(\ell)$ illustrates the generalization of the recursive algorithm from Section 3 of \cite{MR1781453} outputting the digit-reversal permutation. According to that section, performing the bit-reversal permutation on a list can be done recursively first by splitting the list into two halves by ``uninterleaving," which is analogous to unbundling in our context, and applying the procedure to the two halves and then pasting together the resulting lists.
\end{remark}

\section{Number of Inversions}\label{sec:NumberOfInversions}

In this section, we look at the number of inversions that is possible in the stable configuration. Each individual chip-firing preserves the order. This makes it interesting to study when the order reverses. In particular, we are interested in the largest number of inversions possible in the stable configuration.

We now show that this permutation $Z_k(\ell) = R_k(\ell)$ has the largest possible number of inversions. 

\begin{theorem}\label{thm:AlgoInversionProof}
If we start with $k^{\ell}$ labeled chips at the root of a $k$-art tree, the permutation $Z_k(\ell)$ has the maximum possible number of inversions among all permutations corresponding to stable configurations. This number of inversions is
\[\frac{k^{2\ell} -  \ell k^{\ell+1}+(\ell-1)k^{\ell}}{4}.\]
\end{theorem}
\begin{proof}
We use induction on $\ell$. When $\ell=0$, all stable configurations are the same; thus, $Z$ supplies the maximum number of inversions. Moreover, $Z_k(0) = 0$, which matches the expression we are trying to prove. We start with the first statement.

Suppose for the sake of induction that for any whole number $\ell$, the configuration $Z_k(\ell-1)$, when given $k^{\ell-1}$ labeled chips at the root, gives us a stable configuration with the maximum possible number of inversions. We now prove the inductive step.

Consider any way of dispersing the $k^{\ell}$ chips on the root vertex to the $k$ children. Let set $S_m$ denote the set of chips going to the $m$th leftmost child. Let $L_m$ be the list of chips in $S_m$ but in increasing order. Let $w$ be the string resulting from concatenating lists $L_1, L_2, ..., L_{k}$ in the given order. We claim that for each $m \in [k]$ and $m' > m$, the $i$th chip in set $L_{m}$ is part of at most $i-1$ inversions in the string $w$ that consist of a chip in $L_m$ and a chip in $L_{m'}$. If this were not the case, we could denote the value of the $i$th chip in set $L_{m}$ as $k_0+1$, and it would follow that there are more chips in $S_{m'} \cap [k_0]$ than there are in $S_m \cap [k_0].$ This cannot happen as each chip in $S_{m'} \cap [k_0]$ was dispersed from the root with a smaller chip sent to the $m$th vertex. Furthermore, we observe that, since $w$ is a concatenation of lists $L_1, L_2, ..., L_{k}$ and each $L_j$ is a list of increasing chips, any inversion in $w$ consists of a chip in $L_m$ and a chip in $L_{m'}$ for some $m, m' \in [k]$ such that $m \neq m'$.

On the other hand, the unbundling strategy at the root ensures that for each $m, m' \in [k]$ and $i \in [k^{\ell-1}]$ such that $m' > m$, the $i$th largest chip in $L_m$ is part of exactly $i-1$ inversions consisting of a chip in $L_m$ and a chip in $L_{m'}.$ To see this, we observe that in this setup for any $j \in [k]$, $S_j = \{j, j+k, j+2k, ..., j+k^{\ell-1}\}$. We obtain that the $i$th largest chip in $L_m$ is $(i-1)k+m.$ We obtain that $(i-1)k+m$ is larger than $m', k+m', ..., (i-2)k + m'$ in $L_m$ which also appear right of $(i-1)k+j$ in $w.$ 

Doing some computation we find that for any $m, m' \in [k]$ such that $m' > m$, the total number of inversions involving a chip in $L_m$ and one in $L_{m'}$ is $\sum_{i=1}^{k^{\ell-1}} (i-1)  = \frac{(k^{\ell}-1)k^{\ell-1}}{2}$. Now multiply that by the $\binom{k }{2}$, the number of pairs $m, m' \in [k]$ such that $m' > m$, we obtain that $w$ has
\[\binom{k }{ 2} \frac{(k^{\ell}-1)k^{\ell-1}}{2} = \frac{k^{2\ell} - k^{2\ell-1} - k^{\ell+1}+k^{\ell}}{4}\]
inversions.

By the inductive hypothesis, we know that when the firing procedure above, corresponding to the stable configuration $Z_{k-1}(\ell)$, gets applied to the tree rooted at the $m$th child from the left, we will get a stable configuration that has the largest number of inversions when viewed as a permutation of $L_m$ with $[k^{\ell-1}]$ chips. By the inductive hypothesis, each subtree generates 
\[\frac{k^{2\ell-2} -  (\ell -1) k^{\ell} + (\ell-2)k^{\ell-1}}{4}\]
inversions.
Thus, the total maximum number of inversions is 
\[k\cdot \frac{k^{2\ell-2} -  (\ell -1) k^{\ell} + (\ell-2)k^{\ell-1}}{4} + \frac{k^{2\ell} - k^{2\ell-1} - k^{\ell+1}+k^{\ell}}{4}.\]
After collecting the like terms, we get the desired result.
\end{proof}

\begin{example}
In Example~\ref{ex:unbundling} we saw that $Z_2(3)$ permutations is 15372648. This permutation has $4^{3-1}-(3+1)2^{3-2} = 8$ inversions. 
\end{example}

\begin{remark}
    When $k=2$ (i.e., we construct $Z$  from $2^{\ell}$ chips at the root of a binary tree) we obtain that the maximum number of inversions is $4^{\ell-1}-(\ell+1)2^{\ell-2}$, which is the same sequence as A100575 in OEIS \cite{oeis}, which describes half of the number of permutations of $1, 2, ..., n+1$ with two maxima and starts as: 0, 1, 8, 44, 208, 912, 3840, and 15808.
\end{remark}

\section{The Longest Decreasing Subsequence in \texorpdfstring{$Z_k(\ell)$}{Zkell}}\label{sec:LongestDecreasingSubseq}

As each individual firing preserves the order, it is less surprising to get long increasing subsequences in the stable configuration than long decreasing ones. Now, we study the longest decreasing subsequence in $Z_k(\ell).$ For this chapter we will use $Z_k'(\ell)$, the permutation $Z_k(\ell)$ but with $1$ subtracted from each term. In other words, since $Z_k(\ell) = R_k(\ell)$, we have $Z_k'(\ell) = R_k'(\ell).$

We start by discussing palindromic subsequences, which will be useful later. Consider a sequence of $k$-ary strings, each consisting of $\ell$ digits. We call such sequence \textit{palindromic} if the $i$th term from the beginning is the reversal of the $i$th term from the end.

\begin{lemma}
\label{lem:palindromic}
    Consider a palindromic sequence of $k$-ary strings each with $\ell$ digits. If the values of terms in this sequence are decreasing, then the values of this sequence form a subsequence of $Z_k'(\ell)$. If there is a subsequence of $Z_k'(\ell)$ that is palindromic when converted to $k$-ary strings, each with $\ell$ digits, then it is decreasing.
\end{lemma}

\begin{proof}
    Suppose our $k$-ary sequence of strings has decreasing values. If we reverse digits in every term, from the fact that our sequence is palindromic, we get our sequence in reverse order. Thus, digit-reversal makes the values in our sequence increase; thus, from Proposition~\ref{prop:k-itReversedPermutation}, it is a subsequence of $Z_k'(\ell)$. The second statement is proved similarly. As previously observed, the sequence is palindromic; reversing the digits in every term yields the sequence in reversing order. Since the original sequence was in $Z_k'(\ell)$, in which elements are in reflected lexicographic order, the transformed sequence is from greatest to least lexicographic order. Because of this, and since the transformed sequence is the sequence in reversed order, the original sequence is increasing in lexicographic order.
\end{proof}

Lemma~\ref{lem:palindromic} allows us to build decreasing subsequences in $Z_k'(\ell)$. The sequences we consider are subsequences corresponding to stable configurations on $k$-ary trees when we fire $k^\ell$ chips from the root. Given $k$ and $\ell$, we call a sequence of natural numbers \textit{palindromic}, if after representing each number as a $k$-ary string of length $\ell$, we get a palindromic sequence of strings.

\begin{proposition}
\label{prop:FromellToell+2}
    Given a decreasing palindromic subsequence in $Z_k'(\ell)$ of length $d$ with no zero terms, there exists a decreasing, palindromic subsequence in $Z_k'(\ell+2)$ of length $kd+k-1$ with no zero terms.
\end{proposition}

\begin{proof}
    We prove this by construction. Suppose we have a decreasing palindromic subsequence $(b_1, b_2, ..., b_d)$  of $Z_k(\ell)'$ of length $d$ whose terms, when written as length-$\ell$ $k$-ary strings have not all the same digits. We build a new sequence $b_1', b_2', b_3', \dots, b_{kd+k-1}'$ of $k$-ary strings representing nonnegative integers in the following manner.

    First for $i \in \{1, 2, ..., d\}$ we set $b'_i = (k-1)b_i0$, i.e.\ the result of prepending a $k-1$ and appending a 0 to the string representing $b_i$. We then set $b'_{d+1}$ to be the string with $k-1$ as a prefix, then $\ell$ consecutive zeros, and then $1$ as a suffix. Then, for $i \in \{d+2, d+3, \ldots, 2d+1\}$ set $b'_i = (k-2) b_{i - d-1} 1$, followed by a string that starts with $(k-2)$ then all zeros, then 1. We continue in this manner, for each $j \in \{k-1, k-2, ..., 1\}$ building groups of $k$-ary strings so that each starts with $j$ and ends with $k-j-1$.  With the exception of the case $j=k-1$, we have $d+1$ elements in the group; the first $d$ elements in the group have $b_i$s in the middle, and the last element in the group is the string starting with $j+1$, has all zeros in the middle, and ends with $k-j-2.$ When $j= 0$, there are only $d$ elements in the group: $0b_1(k-1), 0b_2(k-1), 0b_3(k-1),..., 0b_d(k-1)$. 
    
    We observe that the new sequence is palindromic. Consider $b'_{(d+1)(j-1)+j'}$ which by our construction is the $j'$th element in the $j$th group. We want to show that $b'_{(d+1)(j-1)+j'} = b'_{kd+k - (d+1)(j-1) - j'}$. First, we address the case, $j' \neq d+1$. This means $b_{(d+1)(j-1)+j'}' = (k-j)b_{j'}(j-1)$. Recalling that $b_1, b_2, ..., b_{j'}$ is palindromic, we obtain that the reversal of digits in $b_{j'}$ is $b_{d-j'}$ and hence the reversal of digits in $b_i'$ is $(j-1)b_{d-j'}(k-j).$ This is exactly the $(d-j')$th element of the $(k-j+1)$st group, i.e., $b_{(d+1)(k-j+1) + d- j'}' = b'_{kd+k - (d+1)(j-1) - j'}.$  Now, we address the case $j' = d+1.$ Note that for any $a \in \{1, 2, ..., k-1\}$, we have $b'_{(d+1)(a-1) + d+1} = (k-a) 00...0(a)$, where there are $\ell$ zeros. If we reverse the digits, we obtain $a 00..0(k-a)$ where there are again $\ell$ zeros: $a00...0(k-a) = b'_{(d+1)(k-a-1) + d+1} = b'_{kd + k - ((d+1)(a-1) + d+1)}.$ 
    
    We now want to show that the new sequence is decreasing. Since the sequence $\{b_i\}$ is decreasing so is the sequence $\{ab_ib\}$, where $a$ and $b$ are fixed digits. Thus, each such group is decreasing. We continue each group with the same $a$ and $b$ and all zeros in the middle. This continuation is decreasing, as all elements in $\{b_i\}$ are positive. What is left to show is that the last element of one group is greater than the first element of the next. But this is true as the first digit decreases between the groups. 
    
    By construction, the new sequence $\{b_i'\}$ is palindromic and decreasing and does not contain zero terms. By Lemma~\ref{lem:palindromic}, its values form a subsequence in $Z_k'(\ell+2)$. Its length is $kd+k-1$, which concludes the proof.
\end{proof}

\begin{example}
Consider $Z'_2(2) = R'_2(2) = $ 0, 2, 1, 3. It contains a decreasing subsequence 2, 1. Writing the result in binary and pretending with zeros when necessary, we get the following two strings: 10 and 01. Using the construction described in Proposition~\ref{prop:FromellToell+2}, we get the following sequence of strings: 1100, 1010, 0110, 0101, 0011. They correspond to numbers 12, 10, 6, 5, and 3. They form a decreasing subsequence in $Z'_2(4) = $ 0, 8, 4, 12, 2, 10, 6, 14, 1, 9, 5, 13, 3, 11, 7, 15.
\end{example}

Now we are ready for the theorem giving the exact length of the longest decreasing subsequence in $Z_k(\ell).$

\begin{theorem}\label{thm:ZLongestDecreasingSubsequence}
        For even $\ell \geq 1$, the longest decreasing subsequence in $Z_k(\ell)$ has length $(k+1) k^{\ell/2-1} - 1$. For odd $\ell \geq 1$, the longest decreasing subsequence has length $2 k^{(\ell-1)/2}-1$.
\end{theorem}

\begin{proof}
Suppose the longest decreasing subsequence in $Z_k(\ell)$ has length $d$. We show that the longest decreasing subsequence in $Z_k(\ell+2)$ has a length not more than $kd + k-1$. We prove this for the shifted case $Z'$, which is equivalent.

Consider the longest decreasing subsequence in $Z_k'(\ell+2)$. We represent each term as a $k$-ary string of length $\ell+2$. Now divide the subsequence into $k$ subsequences that each start with the same digit. Consider a substring that starts with $j >0$. As this substring is a part of $Z'$, which is order reversal, the last digits are in non-decreasing order. If our subsequence is decreasing, it is still decreasing after removing the first and the last digits. Therefore, its length is not more than $d+1$. This is because the last term might be zero, and the rest is a positive decreasing subsequence with respect to digit reversal. Thus, the rest has to belong to $Z_k({\ell})$ and cannot be longer than $d$. By similar reasoning, the last subsequence that starts with 0 has a length not more than $d$. Summing up, we get the result. We see that the longest substring follows the same recursion as in our construction in Proposition~\ref{prop:FromellToell+2}.

Now, we derive the formula by induction. We start with $\ell$ odd. For the base case, we observe that for $\ell = 1$, the length of the longest decreasing subsequence in $Z_k(1)$ is 1. Our formula gives us $2 k^{(1 - 1)/2}- 1 =1$: the same value. Thus, we have a base for induction.

Assume the theorem is true for odd $\ell = i$ that is, the length of the longest decreasing subsequence in $Z_k'(i)$ is $2 k^{(i-1)/2}-1$. Then, by the recursion, we can build the corresponding decreasing sequence in $Z_k'({i+2})$ of length $k (2 k^{(i-1)/2}-1) + k -1 =  2 k^{(i+1)/2} - 1$, concluding the induction.

Now, we move to even numbers. Suppose $\ell = 2$, then the corresponding $Z_k'(\ell)$ contains a decreasing subsequence $(k-1)0$, $(k-2)1$, $(k-3)2$, $\ldots$, $0(k-1)$. The length of this sequence is $k$. Also, permutation $Z'$ consists of $k$ increasing blocks. Hence, the longest subsequence cannot have a length of more than $k$. Our formula gives us $(k+1) k^{2/2-1} - 1 =k$: the same value. Thus, we have a base for induction.

Assume the theorem is true for even $\ell = i$, that is, there exists a decreasing subsequence in $Z_k'(i)$ of length $(k+1) k^{i/2} - 1$. Then, by the recursion, we can build a decreasing sequence in $Z_k'({i+2})$ of length $k (k+1) (k^{i/2} - 1) + k - 1 = (k+1) (k^{(i+2)/2} - k) + (k-1) = (k+1) k^{(i+2)/2} -1$, concluding the induction.
\end{proof}
\begin{example}
     The lengths of longest decreasing subsequences in $Z_2(\ell)$ starting from $Z_2(1)$ are respectively 
     \[1,\ 2,\ 3,\ 5,\ 7,\ 11,\ 15,\ 23,\ 31,\ 47,\ \dots.\]
     This sequence is A052955 in \cite{oeis} shifted. Also, the sequence of lengths of longest decreasing subsequences in $Z_3(\ell)$ starting from $Z_3(1)$ is 
     $$1,\ 3,\ 5,\ 11,\ 17,\ 35,\ 53,\ 107,\ 161,\ 323,\ \dots $$ which is sequence A060647 in \cite{oeis} shifted. 
\end{example}

\section{Decreasing Sequences in Stable Configurations}\label{sec:LongestDecreasingSequencesGeneral}
In this section, we will upper and lower bound the number of terms in the longest possible decreasing sequence.

We begin with an upper bound on the number of configurations of the longest possible decreasing subsequence of a stable configuration, which almost immediately follows from Theorem \ref{thm:ZLongestDecreasingSubsequence}

\begin{proposition}\label{prop:LowerBoundOnLongestDecreasing}
    Given $k^{\ell}$ labeled chips at the root of a directed $k$-ary tree, the longest possible decreasing subsequence in a resulting stable configuration has a length of at least $(k+1)k^{\ell/2-1}-1$ if $\ell$ is even and $2k^{(\ell-1)/2}-1$ if $\ell$ is odd.
\end{proposition}
\begin{proof}
    By Theorem \ref{thm:ZLongestDecreasingSubsequence}, we know that the longest decreasing sequence in $Z_k(\ell)$ is $(k+1)k^{\ell/2}-1$ if $\ell$ is even and $2k^{(\ell-1)/2}-1$ if $\ell$ is odd. Since $Z_k(\ell)$ is by definition a possible stable configuration of chips resulting from stabilizing $k^{\ell}$ labeled chips initially at the root of a directed $k$-ary tree, the result follows.
\end{proof}

We now prove an upper bound on the number of terms in the longest possible decreasing subsequence of a stable configuration resulting from stabilizing $k^{\ell}$ labeled chips starting at the root of a $k$-ary tree.

Let us denote as $D_k(\ell)$ the length of the longest decreasing permutation that can occur in a stable configuration on a $k$-ary tree when we start with $k^\ell$ chips.

The fractal structure of chip-firing allows us to bound the longest decreasing subsequences if we know the value of $D_k(\ell)$ for small $\ell$.

\begin{proposition}
\label{prop:longestdecrease}
        If we start with $k^\ell$ labeled chips at the root of a $k$-ary tree, then the longest strictly decreasing permutation pattern in the stable configuration is at most of length $D_k(n)k^{\ell - n}$ for $n \leq \ell$. In other words, 
        \[D_k(\ell) \leq D_k(n)k^{\ell - n}.\]
\end{proposition}

\begin{proof}
    If we take all subtrees with roots on layer $\ell - n+1$, we end up with $k^{\ell -n}$ subtrees, each containing $k^n$ chips. Each subtree can only have in its stable configuration, at most, a strictly decreasing permutation pattern of length $D_k(n)$. Therefore, we can have, at most, a strictly decreasing permutation pattern of length $D_k(n)k^{\ell - n}$. 
\end{proof}

Thus, calculating $D_k(\ell)$ for small $\ell$ will provide a bound for any $\ell$.

\begin{example}
    We have $D_k(1) = 1$ as the stable configuration is on the second layer, which is in increasing order. We have $D_k(2) \leq k$ as the stable configuration consists of $k$ increasing blocks. On the other hand, $Z_k(2)$ contains a subsequence $k(k-1) + 1,  k(k-2) +2, \dots, k$, which has length $k$. Thus, $D_k(2) = k$.
\end{example}

\begin{example}
    We manually calculated that $D_2(3) = 3$ and $D_2(4) = 5$.
\end{example}

The examples and the Proposition~\ref{prop:longestdecrease} imply the following corollary.

\begin{corollary}
\label{cor:longestbound} For $k \geq 2$, if we start with $k^{\ell}$ chips on the root of a $k$-ary tree, the longest possible decreasing subsequence of the stable configuration is at most $k^{\ell-1}$:
\[D_k(\ell) \leq k^{\ell-1}.\]
In addition, for $\ell \geq 4$, if we start with $2^{\ell}$ chips on the root of a binary tree, the longest possible decreasing subsequence of the stable configuration is at most $5\cdot 2^{\ell-4}$.
\end{corollary}

\begin{proof}
By Proposition~\ref{prop:longestdecrease}, plugging in $n=1$ we obtain that $D_k(\ell) \leq D_k(1)k^{\ell-1} = 1 \cdot k^{\ell-1} = k^{\ell-1}.$ Plugging in $k=2$ and $n=4$, we obtain that $D_2(\ell) \leq D_2(4)2^{\ell-4} = 5\cdot 2^{\ell-4}$.
\end{proof}

In sum, Proposition \ref{prop:LowerBoundOnLongestDecreasing} and Corollary \ref{cor:longestbound} tell us that the longest decreasing subsequence in a stable configuration resulting from stabilizing $k^{\ell}$ labeled chips on a $k$-ary directed tree has length polynomial with respects to $k^{\ell}.$

We end this paper with the conjecture.
\begin{conjecture}\label{conj:DirectedkaryDecreasingSubsequence} Consider a directed $k$-ary tree with $k^{\ell}$ labeled chips initially at the root. The longest decreasing subsequences in a resulting stable configuration do not exceed in length the longest decreasing subsequences of $Z_k(\ell).$
\end{conjecture}

We suspect this conjecture is true since permutations, and hence stable configurations of labeled chips, with long decreasing subsequences, have a large number of inversions and since, by Theorem \ref{thm:AlgoInversionProof}, the permutation $Z_k(\ell)$ has the largest possible number of inversions. In addition, our calculations for $D_k(1)$, $D_k(2)$, $D_2(3)$, and $D_2(4)$ agree with the conjecture.

\section{Acknowledgments} 

Thank you to Professor Alexander Postnikov for suggesting the topic of labeled chip-firing on directed trees and helping formulate the proposal of this research problem. This project started during the Research Science Institute (RSI) program. During RSI, many people helped, and we thank Professors David Jerison and Jonathan Bloom for overseeing the progress of the research problem. We thank Professor Alexander Postnikov for helpful discussions. Our appreciation goes to the RSI students and staff for creating a welcoming working environment.

The first and the second authors are financially supported by the MIT Department of Mathematics. The third author was sponsored by RBC Foundation USA.

\newcommand{\etalchar}[1]{$^{#1}$}

\smallskip

\noindent
Ryota Inagaki \\
\textsc{
Department of Mathematics, Massachusetts Institute of Technology\\
77 Massachusetts Avenue, Building 2, Cambridge, Massachusetts, U.S.A. 02139}\\
\textit{E-mail address: }\texttt{inaga270@mit.edu}
\medskip

\noindent
Tanya Khovanova \\
\textsc{
Department of Mathematics, Massachusetts Institute of Technology\\
77 Massachusetts Avenue, Building 2, Cambridge, MA, U.S.A. 02139}\\
\textit{E-mail address: }\texttt{tanyakh@yahoo.com}
\medskip

\noindent
Austin Luo \\
\textsc{
Morgantown High School,\\
109 Wilson Ave, Morgantown, West Virginia, U.S.A. 26501}\\
\textit{E-mail address: }\texttt{austinluo116@gmail.com}
\medskip

\end{document}